\documentclass[11pt,a4paper]{amsart}
\usepackage{amsmath,amssymb,amsthm,mathtools}
\usepackage[margin=2.6cm]{geometry}
\usepackage{booktabs}
\usepackage[labelsep=period]{caption}
\usepackage[colorlinks=true,linkcolor=blue,citecolor=blue,urlcolor=blue]{hyperref}
\theoremstyle{plain}
\newtheorem{theorem}{Theorem}[section]
\newtheorem{proposition}[theorem]{Proposition}
\newtheorem{lemma}[theorem]{Lemma}
\newtheorem{corollary}[theorem]{Corollary}
\newtheorem{conjecture}[theorem]{Conjecture}
\newtheorem{question}[theorem]{Question}

\NewDocumentEnvironment{mainstatement}{m m m +b}{\expandafter\long\expandafter\gdef\csname mainstatement@#2\endcsname{#4}\begin{#1}[#3~\ref{body:#2}]\label{#2}#4\end{#1}}{}
\newif\ifrestating
\newcommand{\mainref}[1]{\ifrestating\ref{body:#1}\else\ref{#1}\fi}
\newcommand{\restate}[2]{\begin{#1}\restatingtrue\label{body:#2}\csname mainstatement@#2\endcsname\end{#1}}
\theoremstyle{definition}
\newtheorem{remark}[theorem]{Remark}
\numberwithin{equation}{section}

\newcommand{\Z}{\mathbb{Z}}
\newcommand{\Q}{\mathbb{Q}}
\newcommand{\C}{\mathbb{C}}
\newcommand{\RR}{\mathbb{R}}
\newcommand{\HH}{\mathbb{H}}
\newcommand{\PP}{\mathbb{P}}
\newcommand{\Rq}{\mathcal{R}}
\newcommand{\Sq}{\mathcal{S}}
\newcommand{\PSL}{\mathrm{PSL}}
\newcommand{\PGL}{\mathrm{PGL}}
\newcommand{\SL}{\mathrm{SL}}
\newcommand{\GL}{\mathrm{GL}}
\newcommand{\hj}[1]{[\![#1]\!]}
\newcommand{\ncl}[1]{\langle\!\langle #1\rangle\!\rangle}
\newcommand{\trp}{^{\mathsf{T}}}
\newcommand{\sm}[4]{\left(\begin{smallmatrix}#1&#2\\#3&#4\end{smallmatrix}\right)}
\DeclareMathOperator{\tr}{tr}

\title[Denominators of $q$-deformed rational numbers]{Cyclotomic factors and irreducibility\\ of the denominators of $q$-deformed rational numbers}
\author{Kengo Miyamoto}
\address{Department of Computer and Information Sciences, Ibaraki University, Ibaraki, 316-8511, Japan.}
\email{kengo.miyamoto.uz63@vc.ibaraki.ac.jp}
\date{}

\begin{document}

\begin{abstract}
For $d\ge2$ the $q$-deformed modular group specialized at a primitive $d$-th root of unity is the triangle group of type $(2,3,d)$.
Using this we determine the fractions $r/s$ for which the $d$-th cyclotomic polynomial divides the denominator $\Sq_{r/s}(q)$ of the $q$-deformed rational number $[r/s]_q$.
They form the orbit of $\infty$ under the normal closure of $\sm1d01$ in $\PSL(2,\Z)$.
This proves a conjecture of Byakuno, Ren and Yanagawa, and shows that the congruences $s\equiv0$ and $r\equiv\pm1$ modulo $d$ characterize the divisibility exactly for $d\le5$.
For $a\in\{2,3,4,6\}$ and $n>5a^2$ prime to $a$ we show that $\Sq_{a/n}(q)$ is irreducible up to cyclotomic factors.
Together with a computer check for small primes, this confirms a conjecture of Kogiso, Ren, Wakui, Yanagawa and the author for every prime $p$ and every $r$ prime to $p$ with $r\equiv\pm a$ or $ar\equiv\pm1\pmod p$.
\end{abstract}

\maketitle

\section{Introduction}\label{sec:intro}

The $q$-integers $1+q+\dots+q^{n-1}$ are the most basic $q$-analogues of integers.
Morier-Genoud and Ovsienko extended them to all rational numbers by replacing the Pascal triangle in the recursion of the Gaussian binomial coefficients with the Farey graph, and they related the resulting polynomials to quiver representations and to the Jones polynomials of rational knots \cite{MO20}.
Before that, Lee and Schiffler had expressed the Jones polynomials of rational links through cluster algebras and continued fractions \cite{LS19}, see also \cite{CS18}, and Kogiso and Wakui had related them to Conway--Coxeter friezes \cite{KW19}.
The $q$-deformed rational numbers are compatible with the action of the modular group \cite{LM21}, and they appear as boundary points of a compactified space of stability conditions on the $2$-Calabi--Yau category of type $A_2$ \cite{BBL23}.
They also appear in the study of Markov numbers \cite{LLS23}, of the Burau representation of the braid group on three strands \cite{MOV24} and of character varieties \cite{JPT26}.

The denominators of $q$-deformed rational numbers are polynomials in $q$ with integer coefficients, and it is natural to ask how they factor over $\Q$.
Evans, Veselov and Winn studied the $q$-deformed rational numbers whose denominators are products of cyclotomic polynomials \cite{EVW26}.
In this paper we determine which cyclotomic polynomials divide the denominators, by identifying the specialization of the $q$-deformed modular group at a root of unity with a triangle group, and we show that for several infinite families the denominators are irreducible up to these cyclotomic factors.

For an irreducible fraction $r/s$ with $s>0$ the $q$-deformed rational number is a rational function $[r/s]_q=\Rq_{r/s}(q)/\Sq_{r/s}(q)$, where $\Rq_{r/s}(q)$ and $\Sq_{r/s}(q)$ are coprime Laurent polynomials with integer coefficients, normalized by $\Sq_{r/s}(q)\in\Z[q]$ and $\Sq_{r/s}(0)=1$ \cite{MO20}, \cite[Section 2]{BRY26}.
We call them the \emph{numerator} and the \emph{denominator} of $[r/s]_q$.
They satisfy $\Rq_{r/s}(1)=r$ and $\Sq_{r/s}(1)=s$, and for $r/s>1$ they are polynomials with nonnegative integer coefficients.
The definition is recalled in Section~\ref{sec:prelim}.
The $q$-deformation is compatible with the action of the modular group $\Gamma=\PSL(2,\Z)$ on $\PP^1(\Q)=\Q\cup\{\infty\}$ through a $q$-deformation of $\Gamma$ \cite{MO20}, \cite{LM21}, which coincides, after projectivization, with the Burau representation of the braid group on three strands up to the sign of the parameter \cite{MOV24}.
Divisibility of Laurent polynomials is understood in $\Z[q^{\pm1}]$.
For $n\ge1$ let $[n]_q=1+q+\dots+q^{n-1}$ be the $q$-integer.
For $d\ge1$ let $\Phi_d(q)$ be the $d$-th cyclotomic polynomial and $\zeta_d$ a primitive $d$-th root of unity.
Recall that $[n]_q=\prod_{1<d\mid n}\Phi_d(q)$.
For a polynomial $f$ with $f(0)\neq0$ we put $f^\vee(q)=q^{\deg f}f(q^{-1})$.

Morier-Genoud and Ovsienko showed that $\Sq_{r/s}(-1)=0$ if and only if $s$ is even \cite[Proposition 1.8]{MO20}.
The analogous statements for the primitive third and fourth roots of unity were obtained in \cite[Corollary 7.2 and Theorem 7.6]{KMRWY25}, and Byakuno, Ren and Yanagawa showed that $\Sq_{r/s}(\zeta_5)=0$ if and only if $5\mid s$ and $r\equiv\pm1\pmod5$ \cite[Corollary 1.2 (1)]{BRY26}.
Thus for $d\le5$ the polynomial $\Phi_d(q)$ divides $\Sq_{r/s}(q)$ if and only if $r$ and $s$ satisfy a congruence condition modulo $d$.
Byakuno, Ren and Yanagawa proposed the following conjecture and proved it for prime $n$ and for $n=4,6,8,9,10,15,25$ \cite[Proposition 5.7]{BRY26}.

\begin{conjecture}[{\cite[Conjecture 1.4]{BRY26}}]\label{conj:BRY}
Let $n\ge2$.
If $\Phi_n(q)$ divides $\Sq_{r/s}(q)$, then $[n]_q$ divides $\Sq_{r/s}(q)$.
In particular, $\Sq_{r/s}(\zeta_n)=0$ implies $s\in n\Z$.
\end{conjecture}

Our first result describes the divisibility by $\Phi_d(q)$ for every $d\ge2$ in terms of the modular group.
Let $R$ and $S$ be the images of $\sm1101$ and $\sm0{-1}10$ in $\Gamma$.
Every irreducible fraction is of the form $\gamma(\infty)$ with $\gamma\in\Gamma$, and the stabilizer of $\infty$ in $\Gamma$ is $\langle R\rangle$.
For $d\ge2$ let $N_d$ be the normal closure of $R^d$ in $\Gamma$.
Since $N_d$ is normal, $\langle R\rangle N_d$ is a subgroup of $\Gamma$.

\begin{mainstatement}{mainthm}{thm:A}{Theorem}
Let $d\ge2$, let $r/s$ be an irreducible fraction with $s>0$, and let $\gamma\in\Gamma$ satisfy $\gamma(\infty)=r/s$.
\begin{enumerate}
\item $\Phi_d(q)$ divides $\Sq_{r/s}(q)$ if and only if $\gamma\in\langle R\rangle N_d$.
\item $\Phi_d(q)$ divides $\Rq_{r/s}(q)$ if and only if $\gamma\in S\langle R\rangle N_d$.
\end{enumerate}
\end{mainstatement}

In other words, $\Phi_d(q)$ divides $\Sq_{r/s}(q)$ if and only if $r/s$ lies in the orbit of $\infty$ under $N_d$.
The proof rests on the following description of the specialization at $q=\zeta_d$ of the $q$-deformed modular group (Theorem~\ref{thm:kernel}).
Substituting $q=\zeta_d$ in the matrices $R_q=\sm q101$ and $S_q=\sm0{-q^{-1}}10$ defines a homomorphism $\rho_{\zeta_d}$ from $\Gamma$ to $\PGL(2,\C)$.
Its kernel is $N_d$, and the stabilizer of $\infty$ in its image is generated by the image of $R$.
Thus the image of $\rho_{\zeta_d}$ is isomorphic to the triangle group of type $(2,3,d)$, which is finite for $d\le5$, a crystallographic group of the Euclidean plane for $d=6$ and a Fuchsian group for $d\ge7$.
The finiteness of the image for $d\le5$ and its infiniteness for $d\ge6$ were proved by Byakuno, Ren and Yanagawa \cite[Theorem 1.1 and Corollary 3.4]{BRY26}, see also \cite{FK14}, \cite{Lec24}, \cite{LLS23} and \cite{TU25}.
Through the correspondence with the Burau representation \cite{MOV24}, the equality $\ker\rho_{\zeta_d}=N_d$ essentially follows from the description of the image of the Burau representation of the braid group on three strands at roots of unity by Funar and Kohno \cite{FK14}, which was corrected by Dlugie \cite{Dlu24} when $3\mid d$.
We give a short proof which, unlike the proof of \cite[Corollary 1.2 (1)]{BRY26}, does not rely on computer calculations and which also determines the stabilizer of $\infty$.
For $d\ge7$ it compares the traces of the two generators and of their product (Lemma~\ref{lem:trace}) with those of a Fuchsian realization of the triangle group.

Theorem~\ref{thm:A} implies Conjecture~\ref{conj:BRY} in general, together with the congruence conditions observed for $d\le5$.

\begin{mainstatement}{maincor}{cor:B}{Corollary}
Let $n\ge2$ and let $r/s$ be an irreducible fraction with $s>0$.
If $\Phi_n(q)$ divides $\Sq_{r/s}(q)$, then $[n]_q$ divides $\Sq_{r/s}(q)$, and moreover $s\equiv0$ and $r\equiv\pm1\pmod n$.
If $\Phi_n(q)$ divides $\Rq_{r/s}(q)$, then $[n]_q$ divides $\Rq_{r/s}(q)$, and moreover $r\equiv0$ and $s\equiv\pm1\pmod n$.
\end{mainstatement}

The congruence condition of Corollary~\ref{cor:B} is sufficient exactly for $d\le5$, and for $d\ge6$ the divisibility is not determined by the classes $\pm(r,s)$ modulo any positive integer (Remark~\ref{rem:noncongruence}).

\begin{mainstatement}{maincor}{cor:C}{Corollary}
Let $d\ge2$.
\begin{enumerate}
\item If $d\le5$, then for every irreducible fraction $r/s$ with $s>0$, $\Phi_d(q)$ divides $\Sq_{r/s}(q)$ if and only if $s\equiv0$ and $r\equiv\pm1\pmod d$.
If moreover $d\le4$, this is equivalent to $d\mid s$.
\item If $d\ge6$, then $\langle R\rangle N_d$ has infinite index in $\Gamma$, and the congruence condition of \textup{Corollary~\mainref{cor:B}} is not sufficient for $\Phi_d(q)\mid\Sq_{r/s}(q)$.
\end{enumerate}
The same statements hold for $\Rq_{r/s}(q)$ with the roles of $r$ and $s$ interchanged.
\end{mainstatement}

For $d=6$ the image of $\rho_{\zeta_6}$ acts by affine transformations in a suitable coordinate, and this yields an explicit formula for the values at a primitive sixth root of unity.
For an irreducible fraction $r/s>1$ we write $r/s=\hj{c_1,\dots,c_k}$ for its Hirzebruch--Jung continued fraction expansion, which is recalled in Section~\ref{sec:prelim}.

\begin{mainstatement}{mainthm}{thm:D}{Theorem}
Let $\zeta$ be a primitive sixth root of unity.
Let $r/s>1$ be an irreducible fraction with Hirzebruch--Jung continued fraction $r/s=\hj{c_1,\dots,c_k}$, and put $C_0=0$ and $C_i=c_1+\dots+c_i$ for $1\le i\le k$.
Then
\[
\Rq_{r/s}(\zeta)=\zeta^{C_k-2k}\sum_{i=0}^{k}(-1)^i\zeta^{-C_i}
\quad\text{and}\quad
\Sq_{r/s}(\zeta)=\zeta^{C_k-2k-1}\sum_{i=1}^{k}(-1)^i\zeta^{-C_i}.
\]
\end{mainstatement}

The second theme of this paper is the irreducibility of the denominators.
Kogiso, Ren, Wakui, Yanagawa and the author proposed the following conjecture and verified it for $p\le739$ \cite{KMRWY25}.

\begin{conjecture}[{\cite[Conjecture 7.9]{KMRWY25}}]\label{conj:KMRWY}
Let $p$ be a prime and let $a$ be an integer prime to $p$.
Then $\Sq_{a/p}(q)$ is irreducible over $\Q$.
\end{conjecture}

For composite denominators cyclotomic factors occur, and Corollary~\ref{cor:B} restricts them.
For $a\ge2$ put
\[
D_a=\{d\ge2\mid d\text{ divides }a-1\text{ or }a+1\}.
\]
Thus $D_2=\{3\}$, $D_3=\{2,4\}$, $D_4=\{3,5\}$ and $D_6=\{5,7\}$.
Let $n$ be prime to $a$.
By Corollary~\ref{cor:B}, a cyclotomic polynomial $\Phi_d(q)$ with $d\ge2$ divides $\Sq_{a/n}(q)$ only if $d\in D_a$ and $d\mid n$, and $\Phi_1$ does not divide $\Sq_{a/n}(q)$ since $\Sq_{a/n}(1)=n$.

\begin{mainstatement}{mainthm}{thm:E}{Theorem}
Let $a\in\{2,3,4,6\}$, and let $n>5a^2$ be an integer prime to $a$.
Put $C_{a,n}=\prod_{d\in D_a,\,d\mid n}\Phi_d$.
Then
\[
\Sq_{a/n}(q)=C_{a,n}(q)\,E_{a,n}(q),
\]
where $E_{a,n}(q)\in\Z[q]$ is irreducible over $\Q$ and no root of $E_{a,n}(q)$ is a root of unity.
Moreover $E_{a,n}(1)=n/C_{a,n}(1)\ge2$.
\end{mainstatement}

\begin{mainstatement}{maincor}{cor:F}{Corollary}
Let $a\in\{2,3,4,6\}$, let $n>5a^2$ be prime to $a$, and let $r$ be an integer with $r\equiv\pm a$ or $ar\equiv\pm1\pmod n$.
Then $\Sq_{r/n}(q)$ is the product of $C_{a,n}(q)$ and a polynomial which is irreducible over $\Q$, has no root of unity as a root, and takes the value $n/C_{a,n}(1)\ge2$ at $q=1$.
In particular \textup{Conjecture~\ref{conj:KMRWY}} holds for every $b\in\{2,3,4,6\}$, every prime $p>5b^2$ and every $r$ with $r\equiv\pm b$ or $br\equiv\pm1\pmod p$.
\end{mainstatement}

Together with the computer check for $p\le739$ in \cite{KMRWY25}, Corollary~\ref{cor:F} confirms Conjecture~\ref{conj:KMRWY} for every prime $p$ and every $r$ prime to $p$ with $r\equiv\pm b$ or $br\equiv\pm1\pmod p$ for some $b\in\{2,3,4,6\}$ (Remark~\ref{rem:primes}).

The starting point of the proof of Theorem~\ref{thm:E} is the identity $(q-1)\Rq_{n/a}(q)=q^NJ(q)-[a]_q$ for $n\equiv\pm1\pmod a$, where $N=\lfloor n/a\rfloor$ and $J$ is a polynomial with $a$ nonzero coefficients, all equal to $1$ (Lemma~\ref{lem:identity}).
Together with $\Sq_{a/n}=\Rq_{n/a}^\vee$, this expresses $\Sq_{a/n}$ through a lacunary polynomial which, for $N\ge a$, has $2a$ nonzero coefficients $\pm1$ and a large gap.
The cyclotomic factors of this polynomial are determined in Lemma~\ref{lem:cyclotomic}, and its irreducibility up to these factors is proved by the method of Ljunggren \cite{Lju60}, in the form developed by Sawin, Shusterman and Stoll \cite{SSS}.
For $a=2$ the polynomial $q^{N+2}+q^N-q-1$ is a quadrinomial, and the factorization of quadrinomials $q^i\pm q^j\pm q^k\pm1$ was determined by Ljunggren \cite{Lju60} and Mills \cite{Mil85}.
The restriction $a\in\{2,3,4,6\}$ enters in two places.
For these $a$ every $n$ prime to $a$ satisfies $n\equiv\pm1\pmod a$, which is needed for Lemma~\ref{lem:identity}, and the combinatorial Lemma~\ref{lem:divisors} on the divisors of $[a]_qJ^\vee$ is proved only for these $a$.
The bound $n>5a^2$ comes from a counting argument (Lemma~\ref{lem:blocks}).

The paper is organized as follows.
Section~\ref{sec:prelim} recalls the definitions.
Section~\ref{sec:cyclotomic} treats the cyclotomic factors.
It determines the specialization of the $q$-deformed modular group at roots of unity and proves Theorem~\ref{thm:A}, Corollaries~\ref{cor:B} and~\ref{cor:C} and Theorem~\ref{thm:D}.
Section~\ref{sec:irreducibility} proves Theorem~\ref{thm:E} and Corollary~\ref{cor:F} and ends with further questions.

\section{Preliminaries}\label{sec:prelim}

Every irreducible fraction $r/s>1$ has a unique Hirzebruch--Jung continued fraction expansion
\[
\frac rs=\hj{c_1,\dots,c_k}=c_1-\cfrac{1}{c_2-\cfrac{1}{\ddots-\cfrac{1}{c_k}}}\quad(c_i\ge2).
\]
Put
\[
R_q=\begin{pmatrix}q&1\\0&1\end{pmatrix},\quad S_q=\begin{pmatrix}0&-q^{-1}\\1&0\end{pmatrix},\quad M_q(c_1,\dots,c_k)=R_q^{c_1}S_q\cdots R_q^{c_k}S_q\in\GL(2,\Z[q^{\pm1}]).
\]
For $r/s=\hj{c_1,\dots,c_k}>1$ the polynomials $\Rq_{r/s}(q)$ and $\Sq_{r/s}(q)$ are defined as the entries of the first column of $M_q(c_1,\dots,c_k)$, and $[r/s]_q=\Rq_{r/s}(q)/\Sq_{r/s}(q)$ \cite{MO20}, \cite[Definition 2.2]{BRY26}.
They belong to $\Z_{\ge0}[q]$ and satisfy $\Rq_{r/s}(1)=r$, $\Sq_{r/s}(1)=s$ and $\Rq_{r/s}(0)=\Sq_{r/s}(0)=1$.
The $q$-deformation is extended to $\PP^1(\Q)$ by
\begin{equation}\label{eq:shift}
[\alpha+1]_q=q[\alpha]_q+1,\quad [0/1]_q=0,\quad [1/0]_q=1/0,
\end{equation}
and for an arbitrary irreducible fraction $r/s$ with $s>0$ the polynomials $\Rq_{r/s}(q)$ and $\Sq_{r/s}(q)$ are defined as the coprime Laurent polynomials with integer coefficients such that $[r/s]_q=\Rq_{r/s}(q)/\Sq_{r/s}(q)$, $\Sq_{r/s}(q)\in\Z[q]$ and $\Sq_{r/s}(0)=1$ \cite[Section 2]{BRY26}.
It follows from \eqref{eq:shift} that $\Sq_{r/s+n}=\Sq_{r/s}$ for $n\in\Z$.
If $f\in\Z[q]$ is monic with $f(0)\neq0$ and $g\in\Z[q]$, then $f$ divides $g$ in $\Z[q^{\pm1}]$ if and only if $f$ divides $g$ in $\Z[q]$, and this holds if and only if $f$ divides $g$ in $\Q[q]$.

We keep the notation $\Gamma$, $R$, $S$ and $N_d=\ncl{R^d}$ of the introduction.
If $\gamma=\pm\sm abce$, then $\gamma(\infty)=a/c$.
Since the stabilizer of $\infty$ in $\Gamma$ is $\langle R\rangle$, we have $\gamma'(\infty)=\gamma(\infty)$ if and only if $\gamma'\in\gamma\langle R\rangle$.
For $r/s=\hj{c_1,\dots,c_k}>1$ the element $R^{c_1}S\cdots R^{c_k}S$ maps $\infty$ to $r/s$ \cite[(2.1)]{BRY26}.
The group $\Gamma$ has the presentation $\langle S,R\mid S^2,(SR)^3\rangle$ \cite{Ser80}.
For $N\ge1$ let $\Gamma(N)$ be the kernel of the reduction $\Gamma\to\PSL(2,\Z/N\Z)$, where $\PSL(2,\Z/N\Z)=\SL(2,\Z/N\Z)/\{\pm1\}$.

Let $I$ be the identity matrix.
Then $S_q^2=-q^{-1}I$.
The matrix $S_qR_q=\sm0{-q^{-1}}q1$ has trace $1$ and determinant $1$, and the Cayley--Hamilton theorem gives $(S_qR_q)^3=-I$.
Therefore $R\mapsto R_q$ and $S\mapsto S_q$ define a homomorphism $\rho$ from $\Gamma$ to $\PGL(2,\Q(q))$.
For $\zeta\in\C^\times$ they also define a homomorphism
\[
\rho_\zeta\colon\Gamma\to\PGL(2,\C),\quad R\mapsto R_\zeta,\ S\mapsto S_\zeta,
\]
where $R_\zeta$ and $S_\zeta$ are the specializations at $q=\zeta$.
For a word $w$ in $R^{\pm1}$ and $S^{\pm1}$ we write $M_w(q)$ for the corresponding product of $R_q^{\pm1}$ and $S_q^{\pm1}$.
Its determinant is a power of $q$, and $\rho_\zeta$ maps the element of $\Gamma$ represented by $w$ to the class of $M_w(\zeta)$.
The group $\PGL(2,\C)$ acts on $\PP^1(\C)=\C\cup\{\infty\}$ by M\"obius transformations.

\begin{lemma}\label{lem:column}
Let $r/s$ be an irreducible fraction with $s>0$, let $\gamma\in\Gamma$ satisfy $\gamma(\infty)=r/s$, and let $w$ be a word representing $\gamma$.
Then the first column of $M_w(q)$ equals $u\,(\Rq_{r/s}(q),\Sq_{r/s}(q))\trp$ for some $u\in\{\pm q^m\mid m\in\Z\}$.
\end{lemma}

\begin{proof}
Let $w'$ be another word representing $\gamma$.
Since $\rho$ is well defined, we have $M_{w'}(q)=c\,M_w(q)$ for some $c\in\Q(q)^\times$.
Comparing determinants yields $c^2=q^j$ for some $j\in\Z$.
It follows that $j$ is even and $c=\pm q^{j/2}$, whence the first column of $M_{w'}(q)$ is a unit multiple of that of $M_w(q)$.
Next let $\gamma'\in\Gamma$ satisfy $\gamma'(\infty)=\gamma(\infty)$.
Then $\gamma'=\gamma R^l$ for some $l\in\Z$, and $R_q^l$ maps $(1,0)\trp$ to $(q^l,0)\trp$.
Therefore it suffices to prove the lemma for one element $\gamma$ and one word $w$ representing it.
Choose $n\ge0$ with $r/s+n>1$, write $r/s+n=\hj{c_1,\dots,c_k}$, and put $w=R^{-n}R^{c_1}S\cdots R^{c_k}S$.
Let $(A,C)\trp$ be the first column of $M_w(q)=R_q^{-n}M_q(c_1,\dots,c_k)$.
By the definition of $[r/s+n]_q$ and by \eqref{eq:shift}, we have $A/C=[r/s]_q=\Rq_{r/s}/\Sq_{r/s}$.
Since $\det M_w(q)$ is a unit of $\Z[q^{\pm1}]$, the Laurent polynomials $A$ and $C$ are coprime.
The ring $\Z[q^{\pm1}]$ is a unique factorization domain whose units are $\pm q^m$, and $\Rq_{r/s}$ and $\Sq_{r/s}$ are coprime as well.
Therefore $(A,C)=u\,(\Rq_{r/s},\Sq_{r/s})$ for a unit $u$.
\end{proof}

\begin{lemma}\label{lem:zero}
Let $d\ge2$ and let $\zeta$ be a primitive $d$-th root of unity.
Let $r/s$ and $\gamma$ be as in \textup{Lemma~\ref{lem:column}}.
\begin{enumerate}
\item $\Phi_d(q)$ divides $\Sq_{r/s}(q)$ if and only if $\Sq_{r/s}(\zeta)=0$, and this holds if and only if $\rho_\zeta(\gamma)(\infty)=\infty$.
\item $\Phi_d(q)$ divides $\Rq_{r/s}(q)$ if and only if $\Rq_{r/s}(\zeta)=0$, and this holds if and only if $\rho_\zeta(\gamma)(\infty)=0$.
\end{enumerate}
\end{lemma}

\begin{proof}
Since $\Phi_d$ is the minimal polynomial of $\zeta$ over $\Q$, a Laurent polynomial with integer coefficients vanishes at $\zeta$ if and only if it is divisible by $\Phi_d$.
Let $w$ be a word representing $\gamma$.
By Lemma~\ref{lem:column} we have $M_w(\zeta)(1,0)\trp=u(\zeta)\,(\Rq_{r/s}(\zeta),\Sq_{r/s}(\zeta))\trp$ with $u(\zeta)\neq0$.
Since $M_w(\zeta)$ is invertible, it follows that $(\Rq_{r/s}(\zeta),\Sq_{r/s}(\zeta))\neq(0,0)$ and that $\rho_\zeta(\gamma)(\infty)$ is the point $[\Rq_{r/s}(\zeta):\Sq_{r/s}(\zeta)]$ of $\PP^1(\C)$.
\end{proof}

\section{Cyclotomic factors}\label{sec:cyclotomic}

\subsection{Specialization at roots of unity}\label{sec:triangle}

Let $d\ge2$.
The quotient $\Delta_d=\Gamma/N_d$ has the presentation $\langle s,t\mid s^2,(st)^3,t^d\rangle$, where $s$ and $t$ are the images of $S$ and $R$.
This is the triangle group of type $(2,3,d)$.
For $d=2,3,4,5$ it is finite and isomorphic to $S_3$, $A_4$, $S_4$ and $A_5$, of order $6$, $12$, $24$ and $60$ respectively \cite{CM80}.
For $d=6$ it is the group of orientation-preserving symmetries of the regular tessellation of the Euclidean plane by equilateral triangles, and its commutator subgroup is free abelian of rank $2$ \cite{CM80}, \cite{Mag74}.
For $d\ge7$ it is infinite and is realized as a Fuchsian group \cite{Mag74}, \cite{Bea83}.

Let $\zeta$ be a primitive $d$-th root of unity, and put
\[
P_\zeta=\{\gamma\in\Gamma\mid\rho_\zeta(\gamma)(\infty)=\infty\}.
\]
Since $P_\zeta$ is the inverse image under $\rho_\zeta$ of the stabilizer of $\infty$, it is a subgroup of $\Gamma$, and it contains $R$ and $\ker\rho_\zeta$.

\begin{lemma}\label{lem:basic}
Let $\zeta$ be a primitive $d$-th root of unity.
Then $\rho_\zeta(R)$ has order $d$, and $N_d\subset\ker\rho_\zeta$.
The fixed points of $\rho_\zeta(R)$ are $\infty$ and $1/(1-\zeta)$.
The group $\rho_\zeta(\Gamma)$ fixes a point of $\PP^1(\C)$ if and only if $d=6$.
Moreover $\rho_\zeta(\Gamma)$ is not abelian.
\end{lemma}

\begin{proof}
For $j\ge1$ we have $R_\zeta^j=\sm{\zeta^j}{[j]_\zeta}01$ with $[j]_\zeta=(\zeta^j-1)/(\zeta-1)$.
This matrix is scalar if and only if $\zeta^j=1$, that is, if and only if $d\mid j$.
It follows that $\rho_\zeta(R)$ has order $d$ and that $R^d\in\ker\rho_\zeta$.
Since $\ker\rho_\zeta$ is normal, we obtain $N_d\subset\ker\rho_\zeta$.
The transformation $\rho_\zeta(R)$ is $z\mapsto\zeta z+1$, whose fixed points are $\infty$ and $1/(1-\zeta)$.
The transformation $\rho_\zeta(S)$ is $z\mapsto-\zeta^{-1}/z$, which maps $\infty$ to $0$ and $1/(1-\zeta)$ to $1-\zeta^{-1}$.
Now $(1-\zeta^{-1})(1-\zeta)=2-\zeta-\zeta^{-1}$, and this equals $1$ if and only if $\zeta+\zeta^{-1}=1$, that is, if and only if $d=6$.
Thus $\rho_\zeta(S)$ fixes $1/(1-\zeta)$ if and only if $d=6$.
A common fixed point of $\rho_\zeta(\Gamma)$ is a fixed point of $\rho_\zeta(R)$, and $\infty$ is not fixed by $\rho_\zeta(S)$.
Therefore $\rho_\zeta(\Gamma)$ fixes a point of $\PP^1(\C)$ if and only if $d=6$.
Finally, if $\rho_\zeta(S)$ commuted with $\rho_\zeta(R)$, then it would preserve the set $\{\infty,1/(1-\zeta)\}$ of fixed points of $\rho_\zeta(R)$.
This is not the case, since $\rho_\zeta(S)(\infty)=0$.
Thus $\rho_\zeta(\Gamma)$ is not abelian.
\end{proof}

\begin{theorem}\label{thm:kernel}
Let $d\ge2$ and let $\zeta$ be a primitive $d$-th root of unity.
Then $\ker\rho_\zeta=N_d$ and $P_\zeta=\langle R\rangle N_d$.
Equivalently, the homomorphism $\bar\rho_\zeta\colon\Delta_d\to\PGL(2,\C)$ induced by $\rho_\zeta$ is injective, and the stabilizer of $\infty$ in $\bar\rho_\zeta(\Delta_d)$ is generated by $\bar\rho_\zeta(t)$.
\end{theorem}

The two formulations are equivalent by the definition of $P_\zeta$.
We prove the second formulation separately for $d\le5$, for $d=6$ and for $d\ge7$.

\begin{proposition}\label{prop:spherical}
Let $2\le d\le5$, and let $\zeta$ be a primitive $d$-th root of unity.
Then the homomorphism $\bar\rho_\zeta\colon\Delta_d\to\PGL(2,\C)$ induced by $\rho_\zeta$ is injective, and the stabilizer of $\infty$ in $\bar\rho_\zeta(\Delta_d)$ is generated by $\bar\rho_\zeta(t)$.
\end{proposition}

\begin{proof}
Let $G=\bar\rho_\zeta(\Delta_d)$.
By Lemma~\ref{lem:basic}, $G$ is a nonabelian quotient of $\Delta_d$ and contains the element $\bar\rho_\zeta(t)$ of order $d$.
Every proper quotient of $S_3$ and of $A_4$ is abelian.
Every proper quotient of $S_4$ is isomorphic to $S_3$, to $C_2$ or to the trivial group, and none of these has an element of order $4$.
The group $A_5$ is simple.
Therefore $G$ is not a proper quotient of $\Delta_d$, that is, $\bar\rho_\zeta$ is injective.

The stabilizer $G_\infty$ of $\infty$ in $G$ is a finite group of affine transformations $z\mapsto az+b$.
Such a group fixes the barycenter of each of its orbits in $\C$.
After a conjugation by a translation, $G_\infty$ becomes a finite subgroup of $\{z\mapsto az\mid a\in\C^\times\}$, and therefore $G_\infty$ is cyclic.
It contains $\bar\rho_\zeta(t)$, whose order is $d$.
Consequently the order of $G_\infty$ is a multiple $md$ of $d$, and $G_\infty$ has an element of order $md$.
For $d=2,3,4,5$ the groups $S_3$, $A_4$, $S_4$ and $A_5$ respectively contain no element of order $md$ with $m\ge2$.
This forces $m=1$, and $G_\infty=\langle\bar\rho_\zeta(t)\rangle$.
\end{proof}

\begin{proposition}\label{prop:euclid}
Let $\zeta$ be a primitive sixth root of unity.
In the coordinate $w=1/(z-\zeta)$ of $\PP^1(\C)$ we have
\begin{equation}\label{eq:affine}
\rho_\zeta(R)(w)=\zeta^{-1}w,\quad \rho_\zeta(S)(w)=-w-\zeta^{-1},
\end{equation}
and $\rho_\zeta(\Gamma)$ is the group $\mathcal{A}=\{w\mapsto\varepsilon w+\lambda\mid\varepsilon\in\mu_6,\ \lambda\in\Z[\zeta]\}$, where $\mu_6$ is the group of sixth roots of unity.
Moreover the homomorphism $\bar\rho_\zeta\colon\Delta_6\to\PGL(2,\C)$ induced by $\rho_\zeta$ is injective, and the stabilizer of $\infty$ in $\bar\rho_\zeta(\Delta_6)$ is generated by $\bar\rho_\zeta(t)$.
\end{proposition}

\begin{proof}
Recall that $\zeta^2=\zeta-1$ and $\zeta^3=-1$.
For $z\in\C$ we have $\zeta z+1-\zeta=\zeta(z-\zeta)$, and taking inverses yields the first formula.
Since $\zeta^{-2}=-\zeta$, we have $-\zeta^{-1}/z-\zeta=-\zeta(z-\zeta)/z$.
We obtain
\[
\frac{1}{-\zeta^{-1}/z-\zeta}=-\zeta^{-1}\,\frac{z}{z-\zeta}=-\zeta^{-1}-\frac{1}{z-\zeta},
\]
which is the second formula.
Thus $\rho_\zeta(\Gamma)\subset\mathcal{A}$.
The element $\rho_\zeta(SR^3)$ is the translation $w\mapsto w-\zeta^{-1}$.
Its conjugates by the powers of $\rho_\zeta(R)$ are the translations by $-\zeta^{-1-j}$ for $j\in\Z$, and these generate the translations by all elements of $\Z[\zeta]$.
Since the linear part $\zeta^{-1}$ of $\rho_\zeta(R)$ generates $\mu_6$, it follows that $\rho_\zeta(\Gamma)=\mathcal{A}$.

We show that $\bar\rho_\zeta$ is injective.
In the abelianization of $\Delta_6$, written additively, the relations $2s=0$, $3s+3t=0$ and $6t=0$ yield $s=3t$.
Thus the abelianization is generated by the image of $t$, and its order is at most $6$.
The composite of $\bar\rho_\zeta$ and the linear part $\mathcal{A}\to\mu_6$ maps $t$ to $\zeta^{-1}$, which generates $\mu_6$.
It follows that this composite induces an isomorphism from the abelianization of $\Delta_6$ onto $\mu_6$, and that its kernel is the commutator subgroup $K$ of $\Delta_6$.
Therefore $\bar\rho_\zeta$ maps $K$ onto the group of translations, which is isomorphic to $\Z[\zeta]\cong\Z^2$.
Since $K\cong\Z^2$ and every surjective homomorphism $\Z^2\to\Z^2$ is injective, $\bar\rho_\zeta$ is injective on $K$.
As the induced map $\Delta_6/K\to\mu_6$ is injective as well, we conclude that $\bar\rho_\zeta$ is injective.

Finally, the point $\infty$ corresponds to $w=0$, and its stabilizer in $\mathcal{A}$ is $\{w\mapsto\varepsilon w\mid\varepsilon\in\mu_6\}$.
This group is generated by $\rho_\zeta(R)$.
\end{proof}

\begin{lemma}\label{lem:galois}
The subgroups $\ker\rho_\zeta$ and $P_\zeta$ depend only on $d$, not on the choice of the primitive $d$-th root of unity $\zeta$.
\end{lemma}

\begin{proof}
Let $\gamma\in\Gamma$ be represented by a word $w$, and write $M_w(q)=\sm ABCD$.
Then $\gamma\in\ker\rho_\zeta$ if and only if $M_w(\zeta)$ is scalar, that is, if and only if $B(\zeta)=C(\zeta)=0$ and $A(\zeta)=D(\zeta)$.
Since $\Phi_d$ is the minimal polynomial of $\zeta$, this holds if and only if $\Phi_d$ divides $B$, $C$ and $A-D$.
Similarly $\gamma\in P_\zeta$ if and only if $C(\zeta)=0$, that is, if and only if $\Phi_d$ divides $C$.
Neither condition involves the choice of $\zeta$.
\end{proof}

\begin{lemma}\label{lem:trace}
Let $\zeta=e^{2\pi i/d}$ and $\eta=e^{\pi i/d}$, and put $X=\eta S_\zeta$ and $Y=\eta^{-1}R_\zeta$.
Then $X,Y\in\SL(2,\C)$ and
\[
\tr X=0,\quad \tr Y=2\cos\frac\pi d,\quad \tr XY=1.
\]
\end{lemma}

\begin{proof}
Since $\eta^2=\zeta$, we have $\det X=\eta^2\det S_\zeta=\eta^2\zeta^{-1}=1$ and $\det Y=\eta^{-2}\det R_\zeta=\eta^{-2}\zeta=1$.
Moreover $\tr X=0$ and $\tr Y=\eta^{-1}(\zeta+1)=\eta+\eta^{-1}=2\cos(\pi/d)$.
The product $XY=S_\zeta R_\zeta=\sm0{-\zeta^{-1}}\zeta1$ has trace $1$.
\end{proof}

\begin{proposition}\label{prop:hyperbolic}
Let $d\ge7$, and let $\zeta$ be a primitive $d$-th root of unity.
Then the homomorphism $\bar\rho_\zeta\colon\Delta_d\to\PGL(2,\C)$ induced by $\rho_\zeta$ is injective, and the stabilizer of $\infty$ in $\bar\rho_\zeta(\Delta_d)$ is generated by $\bar\rho_\zeta(t)$.
\end{proposition}

\begin{proof}
By Lemma~\ref{lem:galois} we may assume that $\zeta=e^{2\pi i/d}$.
Since $\frac12+\frac13+\frac1d<1$, the theory of triangle groups \cite{Mag74}, \cite{Bea83} provides elements $x,y\in\PSL(2,\RR)$ with the following properties.
The element $x$ is a rotation by the angle $\pi$ about a point $P$ of the upper half plane $\HH$, the element $y$ is a rotation by the angle $\pm2\pi/d$ about a point $Q\in\HH$, and $xy$ is a rotation of order $3$.
The assignment $s\mapsto x$, $t\mapsto y$ defines an isomorphism from $\Delta_d$ onto the discrete subgroup $\Delta'=\langle x,y\rangle$ of $\PSL(2,\RR)$, and the stabilizer of $Q$ in $\Delta'$ is $\langle y\rangle$.

A rotation by the angle $\theta$ lifts to an element of $\SL(2,\RR)$ with trace $\pm2\cos(\theta/2)$.
Let $X'$ and $Y'$ be lifts of $x$ and $y$.
Then $\tr X'=0$, $\tr Y'=\pm2\cos(\pi/d)$ and $\tr X'Y'=\pm1$.
Replacing $Y'$ by $-Y'$ and then $X'$ by $-X'$ if necessary, we may assume that $(\tr X',\tr Y',\tr X'Y')=(0,2\cos(\pi/d),1)$.
By Lemma~\ref{lem:trace} this triple coincides with $(\tr X,\tr Y,\tr XY)$.
For a representation of a free group of rank $2$ in $\SL(2,\C)$, the trace of the image of every element is a polynomial, independent of the representation, in the traces of the images of the two generators and of their product (\cite{Gol09}, see also the proof of \cite[Proposition 1.4.1]{CS83}).
Therefore the two representations of the free group on two letters $a,b$ given by $a\mapsto X,\ b\mapsto Y$ and by $a\mapsto X',\ b\mapsto Y'$ have the same character.
Both representations are irreducible.
Indeed, $X$ and $Y$ have no common eigenvector by Lemma~\ref{lem:basic}, and $x$ and $y$ have no common fixed point in $\PP^1(\C)$, since their fixed points are $P,\bar P$ and $Q,\bar Q$.
Irreducible representations with the same character are conjugate \cite[Proposition 1.5.2]{CS83}.
We obtain $g\in\SL(2,\C)$ with $gXg^{-1}=X'$ and $gYg^{-1}=Y'$.

It follows that $\bar\rho_\zeta$ is the composite of the isomorphism $\Delta_d\cong\Delta'$ with conjugation by $g^{-1}$.
In particular $\bar\rho_\zeta$ is injective.
Moreover the stabilizer of $\infty$ in $\bar\rho_\zeta(\Delta_d)$ corresponds to the stabilizer of $g(\infty)$ in $\Delta'$.
Since $\infty$ is fixed by $Y$, the point $g(\infty)$ is fixed by $y$, that is, it is either $Q$ or its complex conjugate $\bar Q$.
The elements of $\Delta'$ commute with complex conjugation, and therefore the stabilizer of $\bar Q$ in $\Delta'$ equals that of $Q$, namely $\langle y\rangle$.
We conclude that the stabilizer of $\infty$ in $\bar\rho_\zeta(\Delta_d)$ is generated by $\bar\rho_\zeta(t)$.
\end{proof}

Propositions~\ref{prop:spherical}, \ref{prop:euclid} and~\ref{prop:hyperbolic} prove Theorem~\ref{thm:kernel}.

\restate{theorem}{thm:A}
\begin{proof}
Let $\zeta$ be a primitive $d$-th root of unity.
By Lemma~\ref{lem:zero}~(1), $\Phi_d(q)$ divides $\Sq_{r/s}(q)$ if and only if $\gamma\in P_\zeta$, and $P_\zeta=\langle R\rangle N_d$ by Theorem~\ref{thm:kernel}.
This proves (1).
Since $\rho_\zeta(S)(\infty)=0$, we have $\rho_\zeta(\gamma)(\infty)=0$ if and only if $\rho_\zeta(S^{-1}\gamma)(\infty)=\infty$.
Thus Lemma~\ref{lem:zero}~(2) shows that $\Phi_d(q)$ divides $\Rq_{r/s}(q)$ if and only if $S^{-1}\gamma\in P_\zeta=\langle R\rangle N_d$, which proves (2).
\end{proof}

\subsection{Congruence conditions}\label{sec:cons}

\restate{corollary}{cor:B}
\begin{proof}
Let $\gamma\in\Gamma$ satisfy $\gamma(\infty)=r/s$, and assume that $\Phi_n$ divides $\Sq_{r/s}$.
Then $\gamma\in\langle R\rangle N_n$ by Theorem~\ref{body:thm:A}.
Let $l$ be a divisor of $n$ with $l>1$.
Since $R^n=(R^l)^{n/l}$, we have $N_n\subset N_l$ and $\gamma\in\langle R\rangle N_l$.
Therefore Theorem~\ref{body:thm:A} yields $\Phi_l\mid\Sq_{r/s}$.
The polynomials $\Phi_l$ with $1<l\mid n$ are distinct monic irreducible polynomials whose product is $[n]_q$.
It follows that $[n]_q$ divides $\Sq_{r/s}$.

Next, $N_n\subset\Gamma(n)$, since the element $R^n$ lies in the normal subgroup $\Gamma(n)$.
Thus $\gamma\equiv\pm R^j\pmod n$ for some $j\in\Z$.
Comparing the first columns yields $(r,s)\equiv\pm(1,0)\pmod n$.

For the numerator the same argument applies, with $S\langle R\rangle N_n\subset S\langle R\rangle N_l$ in place of $\langle R\rangle N_n\subset\langle R\rangle N_l$ and with the fact that the first column of $SR^j$ is $\pm(0,1)\trp$.
\end{proof}

\restate{corollary}{cor:C}
\begin{proof}
(1) Let $d\le5$.
Since $N_d\subset\Gamma(d)$, the reduction modulo $d$ induces a homomorphism from $\Delta_d$ to $\PSL(2,\Z/d\Z)$ with kernel $\Gamma(d)/N_d$.
Its image is not abelian, since $RS\not\equiv\pm SR\pmod d$, and the class of $R$ in it has order $d$.
As in the proof of Proposition~\ref{prop:spherical}, the homomorphism is injective, and therefore $N_d=\Gamma(d)$.
Therefore, by Theorem~\ref{body:thm:A}, $\Phi_d$ divides $\Sq_{r/s}$ if and only if $\gamma\in\langle R\rangle\Gamma(d)$.
If $\gamma\in\langle R\rangle\Gamma(d)$, then $(r,s)\equiv\pm(1,0)\pmod d$ as in the proof of Corollary~\ref{body:cor:B}.
Conversely, write $\gamma=\pm\sm rbse$ and assume that $s\equiv0$ and $r\equiv\varepsilon\pmod d$ with $\varepsilon\in\{\pm1\}$.
Then $re\equiv1$ yields $e\equiv\varepsilon$ and $\gamma\equiv\pm\varepsilon R^{\varepsilon b}\pmod d$, that is, $\gamma\in\langle R\rangle\Gamma(d)$.
If $d\le4$, then $(\Z/d\Z)^\times=\{\pm1\}$, and therefore the condition $r\equiv\pm1$ follows from $d\mid s$ and $\gcd(r,s)=1$.
For the numerator, Theorem~\ref{body:thm:A}~(2) shows that $\Phi_d$ divides $\Rq_{r/s}$ if and only if $\gamma\in S\langle R\rangle\Gamma(d)$, and the same computation with the first column $\pm(0,1)\trp$ of $SR^j$ shows that this holds if and only if $r\equiv0$ and $s\equiv\pm1\pmod d$.

(2) Let $d\ge6$.
Since $\Delta_d$ is infinite and $t$ has order $d$, the index $[\Gamma:\langle R\rangle N_d]=[\Delta_d:\langle t\rangle]$ is infinite.
On the other hand, the subgroup $\langle R\rangle\Gamma(d)$ contains $\langle R\rangle N_d$ and has finite index in $\Gamma$.
Thus there is an element $\gamma\in\langle R\rangle\Gamma(d)$ which does not lie in $\langle R\rangle N_d$.
In particular $\gamma\notin\langle R\rangle$, and $\gamma(\infty)$ is an irreducible fraction $r/s$ with $s>0$.
As in the proof of Corollary~\ref{body:cor:B}, we have $s\equiv0$ and $r\equiv\pm1\pmod d$.
However, $\Phi_d$ does not divide $\Sq_{r/s}$ by Theorem~\ref{body:thm:A}.
For the numerator, the element $S\gamma$ lies in $S\langle R\rangle\Gamma(d)$ but not in $S\langle R\rangle N_d$, the fraction $S\gamma(\infty)=-s/r$ satisfies $-s\equiv0$ and $r\equiv\pm1\pmod d$, and $\Phi_d$ does not divide $\Rq_{-s/r}$ by Theorem~\ref{body:thm:A}~(2).
\end{proof}

\begin{remark}\label{rem:noncongruence}
For $d\ge6$ the divisibility of $\Sq_{r/s}(q)$ by $\Phi_d(q)$ is not determined by the classes $\pm(r,s)$ modulo any positive integer $N$.
Indeed, suppose that it is, let $g\in\Gamma(N)$, and put $h=SR^{-d}S^{-1}\in N_d$.
Since $gh\equiv\pm h\pmod N$ and $\Phi_d$ divides $\Sq_{h(\infty)}$ by Theorem~\ref{body:thm:A}, either $gh(\infty)=\infty$ or $\Phi_d$ divides $\Sq_{gh(\infty)}$.
In the first case $gh\in\langle R\rangle$, and in the second case $gh\in\langle R\rangle N_d$ by Theorem~\ref{body:thm:A}.
Since $h\in N_d$, this gives $\Gamma(N)\subset\langle R\rangle N_d$, contrary to Corollary~\ref{body:cor:C}~(2).
For $\Rq_{r/s}(q)$ and $N\ge2$ the same argument applied to $Sh$ gives $\Gamma(N)\subset S\langle R\rangle N_dS^{-1}$, which again contradicts Corollary~\ref{body:cor:C}~(2).
Here $gSh(\infty)\neq\infty$, since the first column $\pm(-d,1)\trp$ of $Sh$ is not congruent to $\pm(1,0)\trp$ modulo $N$.
The case $N=1$ is excluded by $\Rq_{d/1}=[d]_q$ and $\Rq_{1/1}=1$.
\end{remark}

\subsection{The Euclidean case}\label{sec:six}

Unlike the cases $d\le5$, the divisibility by $\Phi_6$ is not a congruence condition (Remark~\ref{rem:noncongruence}), and unlike the cases $d\ge7$, the group $\rho_\zeta(\Gamma)$ for $d=6$ fixes a point of $\PP^1(\C)$ and acts by affine transformations (Lemma~\ref{lem:basic} and Proposition~\ref{prop:euclid}), which makes the following explicit formula possible.

\restate{theorem}{thm:D}
\begin{proof}
Let $v=(\zeta,1)\trp$.
Since $\zeta^2-\zeta+1=0$ and $\zeta^3=-1$, we have $R_\zeta v=v$ and $S_\zeta v=\zeta v$.
Let $M=M_\zeta(c_1,\dots,c_k)$ be the specialization of $M_q(c_1,\dots,c_k)$ at $q=\zeta$.
Then $Mv=\zeta^kv$ and $\det M=\zeta^{C_k-k}$.
Put $e_1=(1,0)\trp$.
By definition $Me_1=(\Rq_{r/s}(\zeta),\Sq_{r/s}(\zeta))\trp$, and therefore
\begin{equation}\label{eq:unit}
\Rq_{r/s}(\zeta)-\zeta\,\Sq_{r/s}(\zeta)=\det(Me_1,v)=\zeta^{-k}\det(Me_1,Mv)=\zeta^{-k}\det M\,\det(e_1,v)=\zeta^{C_k-2k}.
\end{equation}
In particular the left-hand side of \eqref{eq:unit} is nonzero.
Consequently $\lambda=\Sq_{r/s}(\zeta)/(\Rq_{r/s}(\zeta)-\zeta\Sq_{r/s}(\zeta))$ is the value of the coordinate $w=1/(z-\zeta)$ at the point $\rho_\zeta(\gamma)(\infty)=[\Rq_{r/s}(\zeta):\Sq_{r/s}(\zeta)]$, where $\gamma=R^{c_1}S\cdots R^{c_k}S$.
By \eqref{eq:affine} the transformation $\rho_\zeta(R^{c_i}S)$ is $f_i(w)=-\zeta^{-c_i}w-\zeta^{-c_i-1}$, and $\infty$ corresponds to $w=0$.
Thus
\[
\lambda=f_1\circ\dots\circ f_k(0)=\sum_{i=1}^k\Bigl(\prod_{j<i}\bigl(-\zeta^{-c_j}\bigr)\Bigr)\bigl(-\zeta^{-c_i-1}\bigr)=\zeta^{-1}\sum_{i=1}^k(-1)^i\zeta^{-C_i}.
\]
Combined with \eqref{eq:unit}, this yields $\Sq_{r/s}(\zeta)=\zeta^{C_k-2k}\lambda=\zeta^{C_k-2k-1}\sum_{i=1}^k(-1)^i\zeta^{-C_i}$ and
\[
\Rq_{r/s}(\zeta)=\zeta\,\Sq_{r/s}(\zeta)+\zeta^{C_k-2k}=\zeta^{C_k-2k}\sum_{i=0}^k(-1)^i\zeta^{-C_i}.\qedhere
\]
\end{proof}

Since $\zeta^2=\zeta-1$ and $\zeta^3=-1$, the identity \eqref{eq:unit} can be written as $\Sq_{r/s}(\zeta)+(\zeta-1)\Rq_{r/s}(\zeta)=\zeta^{C_k-2k+2}$.
This refines the statement of \cite[Lemma 4.2 (1)]{JPT26b}, reproduced in \cite[Lemma 4.4]{JPT26}, that the left-hand side is a sixth root of unity.
Related formulas at a primitive sixth root of unity appear in \cite[Lemma 2.1]{LLS23}.
A common eigenvector of $R_\zeta$ and $S_\zeta$, proportional to $v$, was also used in the proof of \cite[Proposition 3.2]{FK14} to triangularize the Burau representation of the braid group on three strands at $-\zeta$, and in the proof of \cite[Theorem 5.1]{BRY26} to determine the traces of the matrices in the group generated by $R_\zeta$ and $S_\zeta$.

\begin{corollary}\label{cor:six}
Let $\zeta$ be a primitive sixth root of unity, let $r/s=\hj{c_1,\dots,c_k}>1$ be an irreducible fraction, and put $C_0=0$ and $C_i=c_1+\dots+c_i$ for $1\le i\le k$.
Then $\Phi_6(q)$ divides $\Rq_{r/s}(q)$ if and only if $\sum_{i=0}^k(-1)^i\zeta^{C_i}=0$, and $\Phi_6(q)$ divides $\Sq_{r/s}(q)$ if and only if $\sum_{i=1}^k(-1)^i\zeta^{C_i}=0$.
\end{corollary}

\begin{proof}
By Theorem~\ref{body:thm:D}, $\Rq_{r/s}(\zeta)=0$ if and only if $\sum_{i=0}^k(-1)^i\zeta^{-C_i}=0$, and $\Sq_{r/s}(\zeta)=0$ if and only if $\sum_{i=1}^k(-1)^i\zeta^{-C_i}=0$.
Since $\zeta^{-1}=\bar\zeta$, complex conjugation shows that these sums vanish if and only if the sums in the statement vanish.
The assertion now follows from Lemma~\ref{lem:zero}.
\end{proof}

For an arbitrary irreducible fraction $r/s$ with $s>0$, the criterion of Corollary~\ref{cor:six} for $\Sq_{r/s}$ can be applied to $r/s+n$ with $n\in\Z$ such that $r/s+n>1$, since $\Sq_{r/s+n}=\Sq_{r/s}$.

\section{Irreducibility of the denominators}\label{sec:irreducibility}

For $a\ge2$ let $D_a=\{d\ge2\mid d\text{ divides }a-1\text{ or }a+1\}$, and for $n$ prime to $a$ put $C_{a,n}=\prod_{d\in D_a,\,d\mid n}\Phi_d$.
For $r/s>1$ the polynomials $\Rq_{r/s}$ and $\Sq_{r/s}$ of Section~\ref{sec:prelim} coincide with those of \cite{KMRWY25}.
Recall that $R_q^j=\sm{q^j}{[j]_q}01$ for $j\ge1$.

\subsection{Two families of polynomials}\label{sec:families}

For $a\ge2$ we put
\[
J_a^+(q)=1+q^2+q^3+\dots+q^a,\quad J_a^-(q)=1+q+\dots+q^{a-2}+q^a.
\]
Both have exactly $a$ nonzero coefficients, all equal to $1$, and $(J_a^+)^\vee=J_a^-$.
For $a=2$ we have $J_2^+=J_2^-=1+q^2$.

\begin{lemma}\label{lem:identity}
Let $a\ge2$ and $N\ge1$.
Let $n=a(N+1)-1$ and $J=J_a^+$, or let $n=aN+1$ and $J=J_a^-$.
Then
\[
(q-1)\Rq_{n/a}(q)=q^NJ(q)-[a]_q\quad\text{and}\quad\Sq_{a/n}(q)=\Rq_{n/a}(q)^\vee.
\]
\end{lemma}

\begin{proof}
If $n=a(N+1)-1$, then $n/a=(N+1)-1/a=\hj{N+1,a}$.
If $n=aN+1$, then $n/a=(N+1)-(a-1)/a=\hj{N+1,2,\dots,2}$ with $a-1$ entries $2$, since $a/(a-1)=\hj{2,\dots,2}$.
The first column of $R_q^aS_q$ is $([a]_q,1)\trp$.
Moreover the first column of the product of $k$ factors $R_q^2S_q$ is $([k+1]_q,[k]_q)\trp$.
Indeed, this holds for $k=1$, and $R_q^2S_q([k+1]_q,[k]_q)\trp=([2]_q[k+1]_q-q[k]_q,[k+1]_q)\trp$ with $[2]_q[k+1]_q-q[k]_q=[k+2]_q$.
Let $T$ be $R_q^aS_q$ in the first case and the product of $a-1$ factors $R_q^2S_q$ in the second case, and let $(R',S')\trp$ be the first column of $T$.
Thus $(R',S')$ is $([a]_q,1)$ in the first case and $([a]_q,[a-1]_q)$ in the second case.
The first column of $R_q^{N+1}S_qT$ is
\[
R_q^{N+1}S_q\begin{pmatrix}R'\\S'\end{pmatrix}=R_q^{N+1}\begin{pmatrix}-q^{-1}S'\\R'\end{pmatrix}=\begin{pmatrix}[N+1]_qR'-q^NS'\\R'\end{pmatrix}.
\]
Since $(q-1)[N+1]_q=q^{N+1}-1$, this yields
\[
(q-1)\Rq_{n/a}(q)=q^N\bigl(q[a]_q+(1-q)S'\bigr)-[a]_q.
\]
In the first case $q[a]_q+1-q=J_a^+(q)$, and in the second case $q[a]_q+(1-q)[a-1]_q=q[a]_q+1-q^{a-1}=J_a^-(q)$.
The second equality $\Sq_{a/n}=\Rq_{n/a}^\vee$ is \cite[Proposition 5.3 (1)]{KMRWY25} applied to $a/n\in(0,1)$, since both normalizations satisfy $\Sq_{r/s+1}=\Sq_{r/s}$.
\end{proof}

In the rest of this subsection $a\ge2$, $N\ge1$, and $n$, $J$ are as in Lemma~\ref{lem:identity}.
We put
\[
F(q)=q^NJ(q)-[a]_q.
\]
Then $\deg F=N+a$ and $F(0)=-1$.

\begin{lemma}\label{lem:cyclotomic}
Suppose that $n\neq a+1$ and $n\neq a^2-1$.
\begin{enumerate}
\item We have $F+F^\vee=(1-q)(q^N-q^{a-1})$ if $J=J_a^+$, and $F+F^\vee=(1-q^{a-1})(q^N-q)$ if $J=J_a^-$.
These polynomials are nonzero.
\item Every common irreducible factor of $F$ and $F^\vee$ is cyclotomic.
\item The product of the cyclotomic irreducible factors of $F$, counted with multiplicity, is $(q-1)\prod_{d\in D_a,\,d\mid n}\Phi_d$ up to sign.
\end{enumerate}
\end{lemma}

\begin{proof}
(1) Since $[a]_q$ is palindromic, we have $F^\vee=J^\vee-q^{N+1}[a]_q$ and $F+F^\vee=q^N(J-q[a]_q)+(J^\vee-[a]_q)$.
For $J=J_a^+$ we have $J-q[a]_q=1-q$ and $J^\vee-[a]_q=J_a^--[a]_q=q^a-q^{a-1}$.
For $J=J_a^-$ we have $J-q[a]_q=1-q^{a-1}$ and $J^\vee-[a]_q=J_a^+-[a]_q=q^a-q$.
The first polynomial vanishes only for $N=a-1$, that is, for $n=a^2-1$, and the second vanishes only for $N=1$, that is, for $n=a+1$.
Thus both are nonzero under our assumption.

(2) A common irreducible factor of $F$ and $F^\vee$ divides $F+F^\vee$.
By (1) it is therefore $q$ or cyclotomic, and $q$ does not divide $F$.

(3) Let $\zeta$ be a root of unity with $F(\zeta)=0$.
Since $F$ has real coefficients and $\bar\zeta=\zeta^{-1}$, we also have $F^\vee(\zeta)=0$, and $\zeta$ is a root of $F+F^\vee$.

Let $J=J_a^+$, and let $\zeta\neq1$ be a root of unity of order $d$.
If $F(\zeta)=0$, then $\zeta^N=\zeta^{a-1}$ by (1).
Conversely, if $\zeta^N=\zeta^{a-1}$, then the identity $q^{a-1}J_a^+-[a]_q=[a-1]_q(q^{a+1}-1)$ yields $F(\zeta)=[a-1]_\zeta(\zeta^{a+1}-1)$.
This implies that $F(\zeta)=0$ if and only if $\zeta^N=\zeta^{a-1}$ and $[a-1]_\zeta(\zeta^{a+1}-1)=0$, that is, if and only if $d\mid N-a+1$ and $d\in D_a$.
Since $a(N-a+1)=n+1-a^2$ and $\gcd(a,d)=1$ for $d\in D_a$, the condition $d\mid N-a+1$ is equivalent to $d\mid n$ for $d\in D_a$.

Let $J=J_a^-$, and let $\zeta\neq1$ be a root of unity of order $d$.
If $F(\zeta)=0$, then $\zeta^{N-1}=1$ or $\zeta^{a-1}=1$ by (1).
If $\zeta^{N-1}=1$, then the identity $qJ_a^--[a]_q=q^{a+1}-1$ yields $F(\zeta)=\zeta^{a+1}-1$.
If $\zeta^{a-1}=1$, then $J_a^-(\zeta)=\zeta$ and $[a]_\zeta=1$, which gives $F(\zeta)=\zeta^{N+1}-1$.
It follows that $F(\zeta)=0$ if and only if either $d\mid a+1$ and $d\mid N-1$, or $d\mid a-1$ and $d\mid N+1$.
Indeed, in each of these two cases one of the two formulas above shows that $F(\zeta)=0$.
Since $n=a(N-1)+(a+1)=a(N+1)-(a-1)$ and $\gcd(a,d)=1$ for $d\in D_a$, the condition $d\mid N-1$ is equivalent to $d\mid n$ if $d\mid a+1$, and the condition $d\mid N+1$ is equivalent to $d\mid n$ if $d\mid a-1$.
Therefore $F(\zeta)=0$ if and only if $d\in D_a$ and $d\mid n$.

It remains to show that these roots of $F$ are simple.
The root $1$ is simple, since $F'(1)=\Rq_{n/a}(1)=n$.
A multiple root $\zeta\neq1$ of $F$ is a multiple root of $F^\vee$, since $\bar\zeta=\zeta^{-1}$ is a root of $F$ of the same multiplicity.
Thus it is a multiple root of $F+F^\vee$.
By (1) the roots of unity different from $1$ are simple roots of $F+F^\vee$, except in the case $J=J_a^-$ for the common roots of $q^{a-1}-1$ and $q^{N-1}-1$.
Such a root $\zeta\neq1$ of $F$ satisfies $\zeta^{N+1}=1$ and $\zeta^{N-1}=1$, whence $\zeta=-1$, and $a$ and $N$ are odd.
In this case a direct computation yields $F'(-1)=-(N+a)\neq0$.
\end{proof}

\subsection{Ljunggren's method and irreducibility}\label{sec:irred}

The method of Ljunggren \cite{Lju60} studies a factorization $F=gh$ through the polynomial $G=gh^\vee$, which satisfies $GG^\vee=FF^\vee$.

For a polynomial $f$ of degree at most $2a-1$ we put $\widetilde f(q)=q^{2a-1}f(q^{-1})$.

\begin{lemma}\label{lem:blocks}
Let $a\ge2$, $N\ge5a-2$ and $F=q^NJ-[a]_q$ with $J\in\{J_a^+,J_a^-\}$.
Let $G\in\Z[q]$ satisfy $\deg G=N+a$, $G(0)\neq0$ and $GG^\vee=FF^\vee$.
Then there are polynomials $A,H\in\Z[q]$ of degree at most $2a-1$ such that $G=A+q^{N-a+1}H$, both $A$ and $H$ have exactly $a$ nonzero coefficients, all equal to $\pm1$, $A(0)\neq0$, $\deg H=2a-1$, and
\[
A\,\widetilde H=-[a]_q\,J^\vee.
\]
\end{lemma}

\begin{proof}
Write $G=\sum_ig_iq^i$, $F=\sum_if_iq^i$ and $M=N+a$.
For $t\in\Z$ put $c_t(G)=\sum_ig_ig_{i+t}$.
Then $GG^\vee=\sum_tc_t(G)q^{M+t}$, and the same holds for $F$.
The equality $GG^\vee=FF^\vee$ yields $c_t(G)=c_t(F)$ for all $t$.
In particular $\sum_ig_i^2=c_0(G)=c_0(F)=2a$.

The nonzero coefficients of $F$ are $f_i=-1$ for $0\le i\le a-1$ and $f_{N+e}=1$ for the $a$ exponents $e$ of $J$, which lie in $[0,a]$.
Since $N\ge5a-2\ge2a$, a pair $(i,i+t)$ with $f_if_{i+t}\neq0$ and $t\ge N-a+1$ consists of an exponent $i\le a-1$ and an exponent $i+t\ge N$.
There are $a^2$ such pairs, each contributes $-1$ to $c_t(F)$, and their differences $t$ lie in $[N-a+1,N+a]$.
Therefore
\[
\sum_{t\ge N-a+1}|c_t(F)|=a^2.
\]
Let $I_L=[0,2a-1]$ and $I_H=[N-a+1,N+a]$.
These intervals are disjoint, since $N\ge3a-1$.
If $g_ig_{i+t}\neq0$ and $t\ge N-a+1$, then $i\le M-t\le2a-1$ and $i+t\ge N-a+1$, that is, $i\in I_L$ and $i+t\in I_H$.
Put $\sigma_L=\sum_{i\in I_L}|g_i|$ and $\sigma_H=\sum_{i\in I_H}|g_i|$.
It follows that
\[
a^2=\sum_{t\ge N-a+1}|c_t(G)|\le\sigma_L\sigma_H\le\Bigl(\frac{\sigma_L+\sigma_H}{2}\Bigr)^2\le\Bigl(\frac12\sum_i|g_i|\Bigr)^2\le\Bigl(\frac12\sum_ig_i^2\Bigr)^2=a^2.
\]
Therefore all these inequalities are equalities.
This implies that $\sigma_L=\sigma_H=a$, that $g_i\in\{0,\pm1\}$ for all $i$, and that $g_i=0$ for $i\notin I_L\cup I_H$.
Put $A=\sum_{i\in I_L}g_iq^i$ and $H=\sum_{i\in I_H}g_iq^{i-N+a-1}$.
Then $G=A+q^{N-a+1}H$, the degrees of $A$ and $H$ are at most $2a-1$, and both have exactly $a$ nonzero coefficients $\pm1$.
Moreover $A(0)=g_0\neq0$, and $\deg H=2a-1$ since $g_M\neq0$.

Next, $G^\vee=q^{N-a+1}\widetilde A+\widetilde H$, and
\[
GG^\vee=A\widetilde H+q^{N-a+1}\bigl(A\widetilde A+H\widetilde H\bigr)+q^{2(N-a+1)}H\widetilde A.
\]
The three summands have degrees in $[0,4a-2]$, $[N-a+1,N+3a-1]$ and $[2N-2a+2,2N+2a]$ respectively, and these intervals are disjoint since $N\ge5a-2$.
The polynomial $F$ has the same form with $A_F=-[a]_q$ and $H_F=q^{a-1}J$, and $\widetilde{H_F}=J^\vee$.
Comparing the summands of lowest degree in $GG^\vee=FF^\vee$ therefore yields $A\widetilde H=A_F\widetilde{H_F}=-[a]_qJ^\vee$.
\end{proof}

\begin{lemma}\label{lem:divisors}
Let $a\in\{2,3,4,6\}$ and $J\in\{J_a^+,J_a^-\}$.
The only divisors of $[a]_qJ^\vee$ in $\Z[q]$ which have exactly $a$ nonzero coefficients, all equal to $\pm1$, are $\pm[a]_q$ and $\pm J^\vee$.
\end{lemma}

\begin{proof}
We first list the irreducible factors of $[a]_qJ^\vee$.
We have $[a]_q=\prod_{1<d\mid a}\Phi_d$, $J_2^\pm=\Phi_4$ and
\[
J_6^+=\Phi_4\,(q^4+q^3+1),\quad J_6^-=\Phi_4\,(q^4+q+1).
\]
The cubics $J_3^\pm$ are irreducible, since they are monic with constant term $1$ and have no root $\pm1$.
If a monic quartic in $\Z[q]$ with constant term $1$ and without the term $q^3$ is a product of two monic quadratics, then these are $q^2+bq+c$ and $q^2-bq+c$ with $c=\pm1$, and the coefficient of $q$ in the product is $0$.
The quartics $J_4^-=q^4+q^2+q+1$ and $q^4+q+1$ have the coefficient $1$ at $q$ and no root $\pm1$.
Therefore they are irreducible.
Since $J_4^+=(J_4^-)^\vee$ and $q^4+q^3+1=(q^4+q+1)^\vee$, and since reversal preserves irreducibility, these two quartics are irreducible as well.
Consequently every divisor of $[a]_qJ^\vee$ is $\pm$ a product of some of these factors.

The polynomials $\Phi_d$ with $d\ge2$ are palindromic and $(J_a^+)^\vee=J_a^-$.
Thus the divisors for $J=J_a^-$ are the reversals of those for $J=J_a^+$, and reversal preserves the number of nonzero coefficients and their values.
It suffices to treat $J=J_a^+$, that is, $J^\vee=J_a^-$.

Let $A$ be a divisor of $[a]_qJ_a^-$ with exactly $a$ nonzero coefficients, all equal to $\pm1$.
Then $A(0)=\pm1$ and $\deg A\ge a-1$.
Moreover $A(1)\equiv a\pmod2$, and $0<|A(1)|\le a$ since $1$ is not a root of $[a]_qJ_a^-$.
At $q=1$ the factors take the values $\Phi_2(1)=\Phi_4(1)=2$, $\Phi_3(1)=3$, $\Phi_6(1)=1$ and $J_a^-(1)=a$, and $q^4+q+1$ takes the value $3$.

For $a=2$ we need $|A(1)|=2$, which leaves $\pm\Phi_2=\pm[2]_q$ and $\pm\Phi_4=\pm J_2^-$.
For $a=3$ we need $|A(1)|\in\{1,3\}$.
The value $1$ gives $A=\pm1$, which has only one nonzero coefficient, and the value $3$ gives $\pm\Phi_3=\pm[3]_q$ and $\pm J_3^-$.
For $a=4$ we need $|A(1)|\in\{2,4\}$, which leaves $\pm\Phi_2$, $\pm\Phi_4$, $\pm\Phi_2\Phi_4=\pm[4]_q$ and $\pm J_4^-$, and the first two have degree less than $3$.

For $a=6$ put $v=q^4+q+1$.
Then $[6]_qJ_6^-=\Phi_2\Phi_3\Phi_4\Phi_6v$.
Let $i$ be the number of the factors $\Phi_2$ and $\Phi_4$ which occur in $A$, and let $j$ be that of $\Phi_3$ and $v$.
Then $|A(1)|=2^i3^j$ lies in $\{2,4,6\}$, and $(i,j)$ is $(1,0)$, $(2,0)$ or $(1,1)$.
Together with $\deg A\ge5$ this leaves $\pm[6]_q$, $\pm J_6^-$ and $\pm A$ for the five products $A$ in the following table.
\begin{center}
\begin{tabular}{lcl}
\toprule
$A$ & $A(1)$ & expansion\\
\midrule
$\Phi_2\Phi_4\Phi_6$ & $4$ & $1+q^2+q^3+q^5$\\
$\Phi_3\Phi_4\Phi_6$ & $6$ & $1+2q^2+2q^4+q^6$\\
$\Phi_2v$ & $6$ & $1+2q+q^2+q^4+q^5$\\
$\Phi_2\Phi_6v$ & $6$ & $1+q+q^3+2q^4+q^7$\\
$\Phi_4\Phi_6v$ & $6$ & $1+q^2+q^3+q^4+2q^6-q^7+q^8$\\
\bottomrule
\end{tabular}
\end{center}
The first product has only four nonzero coefficients, and each of the others has a coefficient $2$.
\end{proof}

Sawin, Shusterman and Stoll developed this method for the polynomials $F=q^NU(q^{-1})+V(q)$ with fixed $U,V\in\Z[q]$ \cite[Section 4]{SSS}.
They write $G$ in the same form and compare the terms of low degree.
For the pairs $(U,V)$ which they call robust, they show that $\pm F$ and $\pm F^\vee$ are the only solutions of $GG^\vee=FF^\vee$ for all sufficiently large $N$ \cite[Lemma 4.5]{SSS}.
The polynomial $F=q^NJ-[a]_q$ has this form with $N+a$ in place of $N$ and $(U,V)=(J^\vee,-[a]_q)$, and Lemmas~\ref{lem:blocks} and~\ref{lem:divisors} follow the same pattern.
Lemma~\ref{lem:divisors} plays the role of robustness.
The counting argument of Lemma~\ref{lem:blocks} uses that the coefficients of $F$ are $0$ and $\pm1$, and it gives the explicit bound $N\ge5a-2$, which is linear in $a$.
For $0,1$-polynomials $f(q)q^N+g(q)$, bounds linear in the degrees are due to Filaseta and Matthews \cite{FM04}.

\restate{theorem}{thm:E}
\begin{proof}
Since $(\Z/a\Z)^\times=\{\pm1\}$, we have $n\equiv\pm1\pmod a$.
Therefore $n=a(N+1)-1$ with $J=J_a^+$, or $n=aN+1$ with $J=J_a^-$, for some $N\ge1$.
Since $n>5a^2$, in both cases $N\ge5a$.
Put $F=q^NJ-[a]_q$.
By Lemma~\ref{lem:identity}, $F=(q-1)\Rq_{n/a}$.
Since $n\neq a+1$ and $n\neq a^2-1$, Lemma~\ref{lem:cyclotomic} applies, and we can write $F=\pm(q-1)C_{a,n}B$, where $B\in\Z[q]$ has no root of unity as a root.

We show that $B$ is irreducible.
Since $\deg C_{a,n}\le\sum_{d\in D_a}\varphi(d)\le(a-2)+a$, where $\varphi$ is Euler's function, we have $\deg B\ge(N+a-1)-(2a-2)=N-a+1>0$.
Assume that $B=gh$ with $g$ irreducible and $h$ not constant.
Put $G=(q-1)C_{a,n}\,g\,h^\vee$.
Then $G\in\Z[q]$, $\deg G=\deg F$, $G(0)\neq0$ and $GG^\vee=FF^\vee$.
Let $A$ and $H$ be as in Lemma~\ref{lem:blocks}.
Since $A$ divides $[a]_qJ^\vee$, Lemma~\ref{lem:divisors} yields $A=\varepsilon[a]_q$ or $A=\varepsilon J^\vee$ with $\varepsilon\in\{\pm1\}$.
In the first case $\widetilde H=-\varepsilon J^\vee$, that is, $H=-\varepsilon q^{a-1}J$, and $G=-\varepsilon F$.
In the second case $\widetilde H=-\varepsilon[a]_q$, that is, $H=-\varepsilon q^a[a]_q$, and $G=\varepsilon F^\vee$, since $F^\vee=J^\vee-q^{N+1}[a]_q$.
If $G=\pm F$, then $h^\vee=\pm h$.
For an irreducible factor $u$ of $h$ the polynomial $u^\vee$ divides $h$ as well.
Thus $u$ divides $F$ and $F^\vee$, and therefore $u$ is cyclotomic by Lemma~\ref{lem:cyclotomic}~(2).
This contradicts the choice of $B$.
If $G=\pm F^\vee$, then $(q-1)C_{a,n}\,g\,h^\vee=\pm(q-1)^\vee C_{a,n}^\vee g^\vee h^\vee$.
Since $(q-1)^\vee=-(q-1)$ and $C_{a,n}^\vee=C_{a,n}$, we obtain $g^\vee=\pm g$.
Then $g$ divides $F$ and $F^\vee$, and $g$ is cyclotomic, which is again a contradiction.
Therefore $B$ is irreducible.

Put $E_{a,n}=\pm B^\vee$, where the sign is determined by $E_{a,n}(0)=1$.
By Lemma~\ref{lem:identity} and $C_{a,n}^\vee=C_{a,n}$ we obtain $\Sq_{a/n}=C_{a,n}E_{a,n}$.
The polynomial $E_{a,n}$ is irreducible, and its roots are the inverses of the roots of $B$.
None of them is a root of unity.
Since $\Sq_{a/n}(1)=n$, we have $E_{a,n}(1)=n/C_{a,n}(1)$.
The value $C_{a,n}(1)$ is the product of the primes $p$ over the prime powers $p^k\in D_a$ which divide $n$.
It divides $\prod_pp^{\max\{v_p(a-1),v_p(a+1)\}}$, where $v_p$ is the $p$-adic valuation, and this product divides $a^2-1$.
Since $n>a^2-1$, it follows that $E_{a,n}(1)>1$, and thus $E_{a,n}(1)\ge2$.
\end{proof}

\begin{remark}\label{rem:small}
The proof of Theorem~\ref{body:thm:E} works whenever $\lfloor n/a\rfloor\ge5a-2$, the hypothesis of Lemma~\ref{lem:blocks}.
The conclusion of Theorem~\ref{body:thm:E} does not hold for $n\in\{1,a-1,a+1,a^2-1\}$, since $\Sq_{a/1}=1$, $\Sq_{a/(a-1)}=[a-1]_q$, $\Sq_{a/(a+1)}=[a+1]_q$ and $\Sq_{a/(a^2-1)}=[a-1]_q[a+1]_q$ are products of cyclotomic polynomials (see also \cite[Appendix A.1]{EVW26}).
The first two follow from $\Sq_{r/s+1}=\Sq_{r/s}$ and $\Sq_{1/m}=[m]_q$, the third from \cite[Proposition 3.2]{KMRWY25}, and the last from Lemma~\ref{lem:identity} with $N=a-1$ and the identity $q^{a-1}J_a^+-[a]_q=[a-1]_q(q^{a+1}-1)$.
\end{remark}

\restate{corollary}{cor:F}
\begin{proof}
If $r\equiv a\pmod n$, then $\Sq_{r/n}=\Sq_{a/n}$, and by \cite[Theorem 3.5]{KMRWY25} the same holds if $ar\equiv-1\pmod n$.
By \cite[Propositions 3.2 and 3.4]{KMRWY25} we have $\Sq_{r/n}=\Sq_{a/n}^\vee$ if $r\equiv-a$ or $ar\equiv1\pmod n$.
Since the operation ${}^\vee$ preserves irreducibility, inverts the roots, and fixes $C_{a,n}$ and the value at $1$, the first assertion follows from Theorem~\ref{body:thm:E}.

Let $b\in\{2,3,4,6\}$, let $p>5b^2$ be a prime, and let $r\equiv\pm b$ or $br\equiv\pm1\pmod p$.
Then $p$ is prime to $b$.
Moreover $C_{b,p}=1$, since every element of $D_b$ is at most $b+1<p$.
The first assertion, applied to $a=b$ and $n=p$, shows that $\Sq_{r/p}$ is irreducible.
\end{proof}

\begin{remark}\label{rem:primes}
Since $5b^2\le180$ for $b\in\{2,3,4,6\}$, Corollary~\ref{body:cor:F} and the computer check of Conjecture~\ref{conj:KMRWY} for $p\le739$ reported after \cite[Conjecture 7.9]{KMRWY25} show that Conjecture~\ref{conj:KMRWY} holds for every prime $p$ and every $r$ prime to $p$ with $r\equiv\pm b$ or $br\equiv\pm1\pmod p$ for some $b\in\{2,3,4,6\}$.
\end{remark}

\subsection{Further questions}\label{sec:remarks}

\begin{question}\label{q:C1}
Let $r/s$ be an irreducible fraction with $s>0$, and let $f\in\Z[q]$ be an irreducible factor of $\Sq_{r/s}(q)$ with positive leading coefficient.
If $f$ is not a cyclotomic polynomial, does $f(1)\ge2$ hold?
\end{question}

For a prime $s=p$ Question~\ref{q:C1} is equivalent to Conjecture~\ref{conj:KMRWY}.
Indeed, let $a$ be prime to $p$ with $a\not\equiv\pm1\pmod p$.
Then $\Sq_{a/p}$ has no cyclotomic factor by Corollary~\ref{body:cor:B} and $\Sq_{a/p}(1)=p$.
Since $\Sq_{a/p}$ has nonnegative coefficients and $\Sq_{a/p}(0)=1$, its irreducible factors with positive leading coefficients are positive on $(0,\infty)$.
Their values at $q=1$ are positive integers whose product is $p$, and therefore all of them but one take the value $1$.
Thus $\Sq_{a/p}$ has the property in question if and only if it is irreducible.
For $a\equiv\pm1\pmod p$ we have $\Sq_{a/p}=[p]_q=\Phi_p(q)$, which is irreducible and has no non-cyclotomic factor, and both the property in question and the conclusion of Conjecture~\ref{conj:KMRWY} hold for $a/p$.
Corollary~\ref{body:cor:F} answers the question positively if $s>5b^2$ and $r\equiv\pm b$ or $br\equiv\pm1\pmod s$ for some $b\in\{2,3,4,6\}$.

\begin{remark}\label{rem:other}
Lemmas~\ref{lem:identity}, \ref{lem:cyclotomic} and~\ref{lem:blocks} are stated and proved for every $a\ge2$.
In the proof of Theorem~\ref{body:thm:E}, Lemma~\ref{lem:divisors} is only applied to the polynomial $A$ of Lemma~\ref{lem:blocks}, and both $A$ and $-[a]_qJ^\vee/A=\widetilde H$ have exactly $a$ nonzero coefficients $\pm1$.
Thus the conclusion of Theorem~\ref{body:thm:E} holds for $n\equiv\pm1\pmod a$ with $\lfloor n/a\rfloor\ge5a-2$ whenever every divisor $A$ of $[a]_qJ^\vee$ with this property is $\pm[a]_q$ or $\pm J^\vee$.
The conclusion of Lemma~\ref{lem:divisors} fails for $a=12$, since $q^{16}+q^{13}+q^{12}+q^{10}+q^9+q^8+q^7+q^6+q^5+q^4+q+1$ divides $[12]_qJ_{12}^-$.
We do not know whether the weaker property above holds for all $a\ge2$.
For $n=a(N+1)-b$ prime to $a$ with $2\le b\le a-2$ the computation in the proof of Lemma~\ref{lem:identity} gives $(q-1)\Rq_{n/a}=q^N(q\Rq_{a/b}+(1-q)\Sq_{a/b})-\Rq_{a/b}$.
Here $\Rq_{a/b}$ may have coefficients larger than $1$, as $\Rq_{5/2}=1+2q+q^2+q^3$, and Lemma~\ref{lem:blocks} does not apply.
\end{remark}

\section*{Acknowledgments}
This work was supported by JSPS KAKENHI Grant Number JP24K16885.

\end{document}